\documentclass[11pt,letterpaper,dvips]{article}

 \usepackage{amssymb,latexsym,amsmath,amsthm}

\newtheorem*{thm*}{Theorem A}

\newtheorem{thm}{Theorem}

\newtheorem{dfn}{Definition}

\newtheorem{lemma}{Lemma}

\newtheorem{remark}{Remark}

\newtheorem{cor}{Corollary}

\newtheorem{prop}{Proposition}

\newtheorem*{ack}{Acknowledgment}

\begin{document}

\def\d{ \partial_{x_j} }
\def\Na{{\mathbb{N}}}

\def\Z{{\mathbb{Z}}}

\def\IR{{\mathbb{R}}}

\newcommand{\E}[0]{ \varepsilon}

\newcommand{\la}[0]{ \lambda}

\newcommand{\s}[0]{ \mathcal{S}}

\newcommand{\AO}[1]{\| #1 \| }

\newcommand{\BO}[2]{ \left( #1 , #2 \right) }

\newcommand{\CO}[2]{ \left\langle #1 , #2 \right\rangle}

\newcommand{\R}[0]{ \IR\cup \{\infty \} }

\newcommand{\co}[1]{ #1^{\prime}}

\newcommand{\p}[0]{ p^{\prime}}

\newcommand{\m}[1]{   \mathcal{ #1 }}

\newcommand{ \W}[0]{ \mathcal{W}}

\newcommand{ \A}[1]{ \left\| #1 \right\|_H }

\newcommand{\B}[2]{ \left( #1 , #2 \right)_H }

\newcommand{\C}[2]{ \left\langle #1 , #2 \right\rangle_{  H^* , H } }

 \newcommand{\HON}[1]{ \| #1 \|_{ H^1} }

\newcommand{ \Om }{ \Omega}

\newcommand{ \pOm}{\partial \Omega}

\newcommand{\D}{ \mathcal{D} \left( \Omega \right)}

\newcommand{\DP}{ \mathcal{D}^{\prime} \left( \Omega \right)  }

\newcommand{\DPP}[2]{   \left\langle #1 , #2 \right\rangle_{  \mathcal{D}^{\prime}, \mathcal{D} }}

\newcommand{\PHH}[2]{    \left\langle #1 , #2 \right\rangle_{    \left(H^1 \right)^*  ,  H^1   }    }

\newcommand{\PHO}[2]{  \left\langle #1 , #2 \right\rangle_{  H^{-1}  , H_0^1  }}

 \newcommand{\HO}{ H^1 \left( \Omega \right)}

\newcommand{\HOO}{ H_0^1 \left( \Omega \right) }

\newcommand{\CC}{C_c^\infty\left(\Omega \right) }

\newcommand{\N}[1]{ \left\| #1\right\|_{ H_0^1  }  }

\newcommand{\IN}[2]{ \left(#1,#2\right)_{  H_0^1} }

\newcommand{\INI}[2]{ \left( #1 ,#2 \right)_ { H^1}}

\newcommand{\HH}{   H^1 \left( \Omega \right)^* }

\newcommand{\HL}{ H^{-1} \left( \Omega \right) }

\newcommand{\HS}[1]{ \| #1 \|_{H^*}}

\newcommand{\HSI}[2]{ \left( #1 , #2 \right)_{ H^*}}

\newcommand{\WO}{ W_0^{1,p}}
\newcommand{\w}[1]{ \| #1 \|_{W_0^{1,p}}}

\newcommand{\ww}{(W_0^{1,p})^*}

\newcommand{\Ov}{ \overline{\Omega}}

\date{}

\title{Liouville theorems for stable Lane-Emden systems and biharmonic problems}
\author{Craig Cowan\\
{\it\small Department of Mathematical Sciences}\\
{\it\small University of Alabama in Huntsville}\\
{\it\small 258A Shelby Center}\\
\it\small Huntsville, AL 35899 \\
{\it\small ctc0013@uah.edu} }

\maketitle


\vspace{3mm}

\begin{abstract}    We examine the elliptic system given by
\begin{equation} \label{system_abstract}
 -\Delta u = v^p, \qquad -\Delta v = u^\theta, \qquad \mbox{ in } \IR^N,
\end{equation}
 for $ 1 < p \le \theta$ and the fourth order scalar equation
\begin{equation} \label{fourth_abstract}
\Delta^2 u = u^\theta, \qquad \mbox{in $ \IR^N$,}
\end{equation}  where $ 1 < \theta$.  We prove various Liouville type theorems for positive stable solutions.   For instance we show  there are no positive stable solutions of (\ref{system_abstract}) (resp. (\ref{fourth_abstract})) provided $ N \le 10$ and $ 2 \le p \le \theta$ (resp. $ N \le 10$ and  $1 < \theta$).  Results for higher dimensions are also obtained.

 These results regarding stable solutions on the full space imply various  Liouville  theorems for positive (possibly unstable) bounded solutions of
 \begin{equation} \label{eq_half_abstract}
  -\Delta u = v^p, \qquad -\Delta v = u^\theta, \qquad \mbox{ in } \IR^{N-1},
  \end{equation}  with $ u=v=0$ on $ \partial \IR^N_+$.   In particular there is no positive bounded solution of (\ref{eq_half_abstract}) for any $ 2  \le p \le \theta$ if $ N \le 11$.  Higher dimensional results are also obtained.

 \end{abstract}


\section{Introduction}


In this article we examine the nonexistence of positive classical stable solutions of
 the system given by
\begin{equation} \label{eq}
-\Delta u = v^p, \qquad -\Delta v = u^\theta, \qquad \mbox{ in } \IR^N,
\end{equation}
where  $ 1 <  p \le  \theta$.   We also examine the nonexistence of positive classical stable solutions of the fourth order  equation given by
\begin{equation} \label{fourth}
\Delta^2 u = u^\theta \qquad \mbox{ in } \IR^N,
\end{equation}    where $ \theta>1$.

 We now define the notion of a stable solution and for this we prefer to examine a slight generalization of (\ref{eq}) given by
\begin{equation} \label{eq_2}
-\Delta u =  f(v), \qquad -\Delta v = g(u), \qquad \mbox{ in $\IR^N$,}
\end{equation} where $ f,g$ are positive and increasing  on $(0,\infty)$.

\begin{dfn}
  We say a smooth positive solution $(u,v)$ of  (\ref{eq_2}) is stable provided
there exists $ 0 < \zeta, \chi$ smooth with
\begin{equation} \label{stand}
-\Delta \zeta  = f'(v) \chi, \qquad -\Delta \chi =  g'(u) \zeta \quad \mbox{in $ \IR^N$}.
\end{equation}

\end{dfn}
This definition is motivated from \cite{Mont}, also see (\ref{mont_sta}).

\begin{remark}  \label{equi} Note that the standard notion of a stable positive   solution of $ \Delta^2 u = u^\theta$ in $ \IR^N$, is that
\begin{equation} \label{standard}
\int \theta u^{\theta-1} \gamma^2 \le \int (\Delta \gamma)^2,
\end{equation} for all $ \gamma \in C_c^\infty(\IR^N)$.
 For our approach we prefer to recast (\ref{fourth}) into the framework of (\ref{eq}). So towards this suppose $ 1 < \theta$ and $ 0 < u $ is a smooth solution of (\ref{fourth}). Define $ v:=-\Delta u$. By  \cite{Wei_10} $ v >0$ and hence $(u,v):=(u, -\Delta u)$ is a smooth positive solution of (\ref{eq}) with $p=1$.  One now has two options for the notion of the stability of (\ref{fourth}).  Either one views the equation as a scalar equation and uses the standard notion (\ref{standard}),  when we do this we will say $u$ is a stable solution of (\ref{fourth}) or  we view the solution as a solution of the system and we use the  notion defined in (\ref{stand}),  when we do this we will say $(u,v)$ is a stable solution of (\ref{eq}) with $p=1$.   See Lemma \ref{equivalence} for a relationship between these notions of stability.
\end{remark}

We define some parameters before stating our main results.
Given $ 1 \le p \le \theta$ we define
\[ t_0^-:= \sqrt{ \frac{p \theta (p+1)}{\theta+1}} - \sqrt{  \frac{p \theta (p+1)}{\theta+1} - \sqrt{ \frac{p \theta (p+1)}{\theta+1}}},\]
\[ t_0^+:= \sqrt{ \frac{p \theta (p+1)}{\theta+1}} + \sqrt{  \frac{p \theta (p+1)}{\theta+1} - \sqrt{ \frac{p \theta (p+1)}{\theta+1}}}.\]  Properties of $ t_0^-,t_0^+$: \\
(i) $ t_0^- \le 1 \le t_0^+$ and these inequalities are strict except when $ p=\theta=1$. \\
(ii) $t_0^-$ is decreasing and $ t_0^+$ is increasing in $ z:= \frac{p \theta (p+1)}{\theta+1}$ and $ \lim_{z \rightarrow \infty} t_0^-=\frac{1}{2}$.   \\

\noindent We now state out main theorem.

\begin{thm}   (Lane-Emden System) \label{MAIN}
\begin{enumerate}

\item  Suppose $ 2 \le p \le \theta$ and
\begin{equation} \label{cond_syst}
N <2 + \frac{4( \theta+1)}{p \theta-1} t_0^+.
\end{equation}
  Then there is no positive stable solution of (\ref{eq}).  In particular there is no positive stable solution of (\ref{eq}) for any  $ 2 \le p \le \theta$ if $ N \le 10$; see Remark \ref{computations}.

\item  Suppose $  1 < p \le \theta$, $ 2 t_0^- <p$ and   (\ref{cond_syst}) holds.
Then there is no positive stable solution $(u,v)$ of (\ref{eq}).

\end{enumerate}
\end{thm}

\begin{thm} \label{MAIN_four} (Fourth Order Scalar Equation)  Suppose that $ 1=p < \theta$ and
\begin{equation} \label{four_Extremal}
N < 2 + \frac{4(\theta+1)}{\theta-1} t_0^+.
\end{equation}   Then there is no positive stable solution of (\ref{eq}).    In particular there is no positive stable solution of (\ref{eq}), when $p=1$,  for any $ 1 < \theta$  if $ N \le 10$.

\end{thm}

We now turn our attention to the case of half space.    Consider the Lane-Emden system given by
\begin{equation} \label{eq_half}
\left\{ \begin{array}{rll}
\hfill   -\Delta u    &=& v^p \qquad \;  \mbox{ in } \IR^N_+   \\
\hfill -\Delta v &=&  u^\theta \qquad \; \mbox{ in } \IR^N_+,    \\
\hfill u &=& v =0 \quad  \mbox{ on } \partial \IR^N_+,
\end{array}\right.
\end{equation}   where $ 1 < p \le \theta$.

\begin{remark}    This is an updated version of the original work which contained results only regarding stable solutions on the full space.     All results on the half space,  in particular  Theorem \ref{system_thm_half}, did not appear in the original work.   Since the original work appeared  there have been many very nice  improvements, extensions and or related works.   In \cite{new_3} the range of exponents in Theorem \ref{MAIN_four} is improved.   In \cite{new_1} they examine (\ref{eq})  but without any stability assumptions.  They obtain optimal results regarding the existence versus nonexistence of positive radial solutions of (\ref{eq}).   In \cite{new_2} the problem $ \Delta^2 u = |u|^{p-1} u$ in $ \IR^N$ is examined.  They give a complete classification of stable and finite Morse index solutions (no positivity assumptions).

\end{remark}

\noindent We now state our main theorem.

\begin{thm}  \label{system_thm_half}   (Lane-Emden System in $\IR_+^N$)
\begin{enumerate}
\item  Suppose $ 2 \le p \le \theta$ and
\begin{equation} \label{cond_syst_Half}
N-1 <2 + \frac{4( \theta+1)}{p \theta-1} t_0^+.
\end{equation}
  Then there is no positive bounded  solution of (\ref{eq_half}).  In particular there is no positive bounded solution of (\ref{eq_half}) for any  $ 2 \le p \le \theta$ if $ N \le 11$; see Remark \ref{computations}.

\item  Suppose $  1 < p \le \theta$, $ 2 t_0^- <p$ and   (\ref{cond_syst}) holds.
Then there is no positive bounded solution  of (\ref{eq_half}).

\end{enumerate}
\end{thm}

 \begin{remark}  \label{computations}

 We are interested in obtaining lower bounds on the right hand side of (\ref{cond_syst}), in the case where $ 2 \le p \le \theta$,  and so we set  $ f(p,\theta):= \frac{4(\theta+1)}{p \theta-1} t_0^+$.  We rewrite $f$ using the change   of variables $ z= \frac{p \theta(p+1)}{\theta+1}$ to arrive at
 \[ \tilde{f}(p,z)=  \frac{4p}{z-p} \left( \sqrt{z}+ \sqrt{ z - \sqrt{z}} \right),\] and the transformed domain is given by
 \[ \mathcal{D}=\{ (p,z): p \ge 2, \; p^2 \le z \le p^2+p \}.\]
 A computer algebra system easily shows that $ \tilde{f}>8$ on $\partial \mathcal{D}$.    Note that $ \partial_p \tilde{f} > 0$ on $ \mathcal{D}$ and so  we have  $ \tilde{f} >8$ on $ \mathcal{D}$ which gives us the desired result.

 \end{remark}

There has been much work done on the existence and nonexistence of positive classical solutions of the Lane-Emden equation given by
\begin{equation} \label{lane_class}
-\Delta u = u^\theta, \qquad \mbox{in $ \IR^N$,}
\end{equation}  for instance see
 \cite{Caf}, \cite{chen}, \cite{gidas},\cite{Gidas}. It is known that there are no positive classical  solutions of (\ref{lane_class}) provided that
\[ 1<  \theta < \frac{N+2}{N-2},\] and in the case of $ N =2$ there is no positive solution for any $ \theta>1$.   We further remark that this is an optimal result.  It is well known that a Liouville theorem related to (\ref{lane_class}) implies apriori estimates of solutions of the same equation on a bounded domain, see \cite{Caf}, \cite{gidas}.
 We remark that this equation originally appeared in astrophysics, where it was a model for the gravitational potential of a Newtonian self-gravitating, spherically symmetric, polytropic fluid.

In \cite{Wang_solo} and \cite{Gui_Ni_Wang} parabolic versions of (\ref{lane_class}) were examined and in particular they were interested in the asymptotics.   They also obtained various properties of solutions of (\ref{lane_class}) and in particular their results easily imply  that for
 $ N \ge 11$ and $ \theta \ge \theta_{JL}$, where $\theta_{JL}$ is the so called Joseph-Lundgren exponent,  there exists positive stable radial solutions of (\ref{lane_class}).

In \cite{farina} the equation
\begin{equation} \label{far}
-\Delta u = |u|^{\theta-1} u, \qquad \mbox{in $ \IR^N$,}
\end{equation} was examined.  They completely classified the finite Morse index solutions of (\ref{far}).   It was shown there exists nontrivial finite Morse index solutions of (\ref{far}) if and only if $ N \ge 11$ and $ \theta \ge \theta_{JL}$.

In \cite{Wei_dong} the nonexistence of stable solutions  of (\ref{fourth}) was examined.     It was shown that there is no positive stable  solutions of (\ref{fourth}) provided either: $ N \le 8$ or   $ N \ge 9$ and $ 1 < \theta < \frac{N}{N-8} + \E_N$ where $ \E_N $ is some positive, but unknown parameter.  We also alert the reader to the work of \cite{Warn},  where the nonexistence of stable solutions of $ \Delta^2 u = f(u)$ in $ \IR^N$ was examined for general nonlinearities $f$.  Many interesting results were obtained.

As mentioned above Liouville theorems are extremely useful for the existence of apriori estimates of solutions on bounded domains.    The nonexistence of nontrivial stable solutions of $ -\Delta u = g(u)$ in $ \IR^N$ is closely related to the regularity of the  extremal solution associated with
\begin{equation*}
(Q)_\lambda \qquad  \left\{
\begin{array}{ll}
-\Delta u =\lambda f(u) &\hbox{in }\Omega \subset \subset \IR^N, \\
u =0 &\hbox{on } \pOm
\end{array}
\right.
\end{equation*}    where $ \lambda $ is a positive parameter and $ f(u)$ is a nonlinearity which is related to $g$.   Here the extremal solution $u^*$ is the minimal solution, ie. smallest in the pointwise sense,  of $ (Q)_{\lambda^*}$ where $ \lambda^*$ is the largest parameter $\lambda$ such that $(Q)_\lambda$ has a weak solution.    The critical fact in proving the regularity of $u^*$ in certain cases is the fact that $ u^*$ is a stable solution of $(Q)_{\lambda^*}$.  See \cite{bcmr,BV,Cabre,CC,advection,CR,EGG,GG,Martel,MP,Nedev}  for results concerning $(Q)_\lambda$.

We now examine some bounded domain analogs of (\ref{eq}) and (\ref{fourth}).
We begin by examining
 \begin{eqnarray*}
(P)_{\lambda,\gamma}\qquad  \left\{ \begin{array}{lcl}
\hfill   -\Delta u    &=& \lambda f(v)\qquad \Omega  \\
\hfill -\Delta v &=& \gamma g(u)   \qquad \Omega,  \\
\hfill u &=& v =0 \qquad \pOm,
\end{array}\right.
  \end{eqnarray*}  where $ \Omega$ is a bounded domain in $ \IR^N$ and where $f,g$ are smooth, positive increasing nonlinearities which are superlinear at $ \infty$.    Set    $ \mathcal{Q}=\{ (\lambda,\gamma): \lambda, \gamma >0 \}$,  \;
$ \mathcal{U}:= \left\{ (\lambda,\gamma) \in \mathcal{Q}: \mbox{ there exists a smooth solution $(u,v)$ of $(P)_{\lambda,\gamma}$} \right\},$   and  set
 $ \Upsilon:= \partial \mathcal{U} \cap \mathcal{Q}$.  Note $\Upsilon$ plays the role of the extremal parameter $ \lambda^*$ in the case of the system.    Using monotonicity one can define
    an extremal solution $(u^*,v^*)$ for each $ (\lambda^*,\gamma^*) \in \Upsilon$.

       To show the regularity of $(u^*,v^*)$ we will need to use the minimality of the minimal solutions
 to obtain added regularity.   In \cite{Mont} a generalization of $(P)_{\lambda,\gamma}$ was examined and various properties we obtained.   One important result was that the minimal solutions $(u,v)=(u_{\lambda,\gamma}, v_{\lambda,\gamma})$ were stable in the sense that there was some nonnegative constant $ \eta$ and $ 0 < \zeta,\chi \in H_0^1(\Omega)$ such that
 \begin{equation} \label{mont_sta}  -\Delta \zeta = \lambda f'(v) \chi + \eta \zeta, \quad -\Delta \chi = \gamma g'(u) \zeta + \eta \chi, \quad \mbox{in $ \Omega$.}
\end{equation}
It is precisely this result that motivates our definition of a stable solution of (\ref{eq}).   Until recently it was not known hot to utilize the stability of solutions to obtain results regarding the regularity of the extremal solutions associated with the system $(P)_{\lambda,\gamma}$,  except in very special cases.  For instance in \cite{craig0} results were obtained in the case of $ f(v)=e^v$, $ g(u)=e^u$.   Very recently, in \cite{craig_lane}, the regularity of the extremal solutions associated with $(P)_{\lambda,\gamma}$ was examined in  the case of $ f(v)=(v+1)^p$ and $ g(u)=(u+1)^\theta$ where $ 1 < p \le \theta$.  It was shown that the associated extremal solutions were bounded provided condition   (\ref{cond_syst}) holds.  In \cite{craig_four} the fourth order problem
 \begin{eqnarray*}
(N)_{\lambda}\qquad  \left\{ \begin{array}{lcl}
\hfill   \Delta^2 u    &=& \lambda f(u)\qquad \quad  \Omega,  \\
\hfill u &=& \Delta u =0 \qquad \pOm,
\end{array}\right.
  \end{eqnarray*}  where $ \Omega $ is a bounded domain in $ \IR^N$ was examined in the case where $ f(u)=e^u$ and $ f(u)=(u+1)^\theta$.   In the case of $ f(u)=e^u$,  the previous best known result was from \cite{craig1} where it was shown that $ u^*$ was bounded provided $ N \le 8$.  In \cite{craig_four} this was improved to $ N \le 10$  but this still falls far short of the expected optimal result of $ N \le 12$ after one considers the results of \cite{DDGM} on radial domains.
  \begin{remark} Two weeks after this work was made available online we received the manuscript \cite{NEW_GELF} where
they also use this new idea of system stability for a scalar equation, see Lemma 2 in \cite{craig_four} or Lemma \ref{stabb} in the current work.  They examined $(N)_\lambda$ and (\ref{fourth}) in the case where the nonlinearity is given by $f(u)=e^u$.  In the case of the bounded domain they obtained an optimal result.
\end{remark}

   In \cite{craig_four} we showed the extremal solution associated with $ (N)_{\lambda}$ in the case of $ f(u)=(u+1)^\theta$ is bounded  provided condition (\ref{four_Extremal}) holds.  Again these were major improvement over the best known previous results, again from \cite{craig1},  but fall short of the expected optimal results, see \cite{DFG} for the radial case.  For more works related to $(N)_\lambda$ see \cite{BG,CDG,CEG,GW}.

  Note that our results in the current work are expected after viewing the recent works on the regularity of the extremal solutions on bounded domains.   It should be noted that even if this is to be expected,  these  results are not straightforward adaptions of the regularity results on bounded domains.  We finish off by mentioning  three related works on systems.   The first two results \cite{fg,fazly_system} deal with elliptic systems and stability.
  The third work, \cite{fazly}, has many results.
     One of the results is the nonexistence of nontrivial stable solutions of
  \[ -\Delta u = |x|^\alpha v^p, \qquad -\Delta v = |x|^\beta u^\theta \qquad \mbox{ in $\IR^N$,}\]  under certain restrictions of the involved parameters.  For this result the methods developed in the current work are extended to handle the case of nonzero $ \alpha$ and $ \beta$.

\subsection{The half space}

We are also interested in Liouville theorems related to positive bounded solutions of (\ref{eq_half}).     There has been much work done on these and related equations see, for instance, \cite{half_1,half_2,half_3,quittner,half_4}.     The best known result to date regarding a Liouville theorem for positive bounded solutions of  (\ref{eq_half}) is given by \cite{quittner}.      Theorem \ref{system_thm_half}  improves this non existence result.

\subsection{A brief outline of the approach}

Here we give a brief outline of the approach we take.  Suppose $ (u,v)$ is a smooth positive stable solution of (\ref{eq}) with $ 2 < p < \theta$.   We begin by showing that stability implies  \[ \sqrt{p \theta} \int u^\frac{\theta-1}{2} v^\frac{p-1}{2} \phi^2 \le \int | \nabla \phi|^2,\] for all $ \phi \in C_c^\infty(\IR^N)$.   Using as a test function $ \phi=v^t \gamma$ where $ \gamma \in C_c^\infty(\IR^N)$ and using the pointwise comparison $ (\theta +1) v^{p+1} \ge (p+1) u^{\theta+1}$ in $ \IR^N$,  which holds without stability; see \cite{souplet_4}, we obtain an inequality roughly of the form
\begin{equation} \label{ten}
 \int u^\theta v^{2t-1} \gamma^2 \le C_t \int v^{2t} | \nabla \gamma|^2,
\end{equation}  for all $ t_0^- <t<t_0^+$.    One also has the following integral estimates, see \cite{Mid}
\begin{equation} \label{twenty}
 \int_{B_R} v^p \le C R^{N-2- \frac{2(p+1)}{p \theta-1}},
\end{equation}
 which also holds without the stability assumption.     As a first attempt we assume that $ t_0^- < \frac{p}{2} < t_0^+$ and so we can take $ 2t =p$ in (\ref{ten}) and  we assume that $ 0 \le \gamma \le 1$ with $ \gamma=1 $ in $ B_R$ and is compactly supported in $ B_{2R}$. We then use (\ref{ten}) and (\ref{twenty}) to see that
 \[ \int_{B_R} u^\theta v^{p-1} \le \frac{C}{R^2} \int_{B_{2R}} v^p \le C R^{ N-4- \frac{2(p+1)}{p \theta-1}},\]
for all $ R >0$.  Provided this exponent is negative then we get a contradiction by sending $ R \rightarrow \infty$.   Note this implies there is no positive stable solution of (\ref{eq}) for $ N \le 4$ for any $ 2 \le p \le \theta$.  Now since we expect $v$ to decay to zero at $\infty$ one would expect to obtain better results if we can choose $ t > \frac{p}{2}$.
To do this we examine the elliptic equation which $ v^p$ satisfies and we use $L^1$ elliptic regularity theory along with (\ref{ten}) and (\ref{twenty}) to obtain integral estimates on $ v^{p \alpha}$ for $1 < \alpha < \frac{N}{N-2}$.  We can iterate this procedure, as long as the range of $t$ in (\ref{ten}) allows,  looking at increasing powers of $ v$ to obtain integral estimates of $v$ for higher powers.  These higher power integral estimates allow one to obtain better results.

\begin{ack} We would like to thank Philippe Souplet for bringing to our attention the work \cite{PHAN}.    In a previous version of this paper we had needed to assume that either: $ 2 \le \theta$ or $ u$ is bounded in Lemma \ref{pointwise}.
\end{ack}

\section{Proof of Theorem \ref{MAIN} and \ref{MAIN_four}.}

We begin with some integral estimates which are valid for any positive solution of (\ref{eq}).

\begin{lemma} \label{Mid_lemma} \emph{\cite{Mid}}   Suppose $(u,v)$ is a positive solution of (\ref{eq}) with $ 1 \le p \le \theta$.  Then
\[ \int_{B_R} v^p \le CR^{N-2- \frac{2(p+1)}{p \theta-1}}, \]
\[ \int_{B_R} u^\theta  \le CR^{N-2- \frac{2(\theta+1)}{p \theta-1}}.\]

\end{lemma}

A crucial ingredient in our proof of Theorem \ref{MAIN} is given by the following pointwise comparisons.

\begin{lemma}  \label{pointwise} \begin{enumerate} \item  \cite{PHAN}, \cite{souplet_4} Suppose that $(u,v)$ is a smooth solution of (\ref{eq}) and $ 1 < p \le \theta$.
  Then
\begin{equation} \label{point}
(\theta +1)  v^{p+1} \ge (p+1) u^{\theta+1} \qquad \mbox{ in $ \IR^N$.}
\end{equation}

\item \cite{Wei_dong} Suppose that $u$ is a smooth stable solution of (\ref{eq}) with  $ 1=p < \theta$.    Then there exists a  smooth positive stable bounded solution of (\ref{eq}), which we denote by $ \tilde{u}$ and which satisfies
\begin{equation} \label{point_four}
(\theta +1)  \tilde{v}^{2} \ge (p+1) \tilde{u}^{\theta+1} \qquad \mbox{ in $ \IR^N$,}
\end{equation}
  where $ \tilde{v}:= - \Delta \tilde{u}>0$.

\end{enumerate}

\end{lemma}

\begin{remark}   When attempting to prove the nonexistence of positive stable solutions $(u,v)$ of (\ref{eq}), in the case $p=1$, we can use the above lemma 2) to  assume that $u$ is bounded.

\end{remark}

The following lemma transforms our notion of a stable solution of (\ref{eq}) into an inequality which allows the use of arbitrary test functions.  A bounded domain version of this was proven in \cite{craig_lane}  but we include the proof here for the readers sake.  We remark that this result was motivated by a similar result in \cite{craig2}.

\begin{lemma} \cite{craig_lane} \label{stabb}

 Let $(u,v)$ denote a stable solution of (\ref{eq_2}).  Then
\begin{equation} \label{second}
 \int \sqrt{f'(v) g'(u)} \phi^2 \le \int | \nabla \phi|^2
\end{equation} for all $ \phi \in C_c^\infty(\IR^N)$.

\end{lemma}

\begin{proof} Let  $(u,v)$ denote a stable solution of (\ref{eq_2}) and so there is some $ 0 < \zeta, \chi $ smooth such that
\[ \frac{-\Delta \zeta}{\zeta} = f'(v) \frac{\chi}{\zeta}, \qquad \frac{-\Delta \chi}{\chi} =  g'(u) \frac{\zeta}{\chi}, \qquad \mbox{ in $ \IR^N$.}\]  Let $ \phi,\psi \in C_c^\infty(\IR^N)$ and multiply the first equation by $ \phi^2$ and the second by $\psi^2$  and integrate over $ \IR^N$ to arrive at
\[ \int  f'(v) \frac{\chi}{\zeta} \phi^2 \le \int | \nabla \phi|^2, \qquad \int  g'(u) \frac{\zeta}{\chi} \psi^2 \le \int | \nabla \psi|^2,\]  where we have utilized the result that  for any sufficiently smooth $ E>0$ we have
\[ \int \frac{-\Delta E}{E} \phi^2 \le \int | \nabla \phi|^2,\] for all $ \phi \in C_c^\infty(\IR^N)$.   We now add the inequalities to obtain
\begin{equation} \label{thing}
 \int ( f'(v)  \phi^2)   \frac{\chi}{\zeta}  + ( g'(u)  \psi^2 )\frac{\zeta}{\chi} \le \int | \nabla \phi|^2 + | \nabla \psi|^2.
\end{equation}   Now note that
\[ 2 \sqrt{ f'(v) g'(u)} \phi \psi \le 2t  f'(v) \phi^2 + \frac{1}{2t}  g'(u) \psi^2,\] for any $ t>0$.  Taking $ 2t = \frac{\chi(x)}{\zeta(x)}$ gives
\[ 2 \sqrt{  f'(v) g(u)} \phi \psi \le ( f'(v)  \phi^2)   \frac{\chi}{\zeta}  + ( g'(u)  \psi^2 )\frac{\zeta}{\chi},\]  and putting this back into (\ref{thing}) gives the desired result after taking $ \phi=\psi$.
\end{proof}

 For integers $  k \ge -1$  define $ R_k:= 2^k R$ for  $ R >0$.

\begin{lemma}  \label{one}  Suppose $ (u,v)$ is a smooth, positive stable solution of (\ref{eq}) satisfying the hypothesis from Lemma \ref{pointwise}. Then for all $ t_0^- < t <t^+_0$ there is some $C_t< \infty$ such that
\[ \int_{B_{R_k}} u^\theta v^{2t-1}  \le \frac{C_t}{2^{2k}R^2} \int_{B_{R_{k+1}}} v^{2t} ,\] for all $ 0<R <\infty$.

\end{lemma}

\begin{proof}  Let $ (u,v)$ denote a positive smooth stable solution of (\ref{eq}).    Let $ \gamma$ denote a smooth cut-off function which is compactly supported in $ B_{R_{k+1}}$ and which is equal to one in $ B_{R_k}$.  Put $ \phi:= v^t \gamma$ into (\ref{second}) to obtain
\[ \sqrt{p \theta} \int v^\frac{p-1}{2} u^\frac{\theta-1}{2} v^{2t} \gamma^2 \le
  t^2 \int v^{2t-2} | \nabla v|^2 \gamma^2  + \int v^{2t} | \nabla \gamma|^2 + 2 t \int v^{2t-1} \gamma \nabla v \cdot \nabla \gamma.\]   We now re-write the left hand side as
  \[ \sqrt{p \theta} \int u^\frac{\theta-1}{2} v^\frac{p+1}{2} v^{2t-1} \gamma^2,\] and we now use the the pointwise bound (\ref{point}) to see the left hand side is greater than or equal to
  \[  \sqrt{ \frac{p \theta (p+1)}{\theta+1}} \int u^\theta v^{2t-1} \gamma^2,\] and hence we obtain
  \[ \sqrt{ \frac{p \theta (p+1)}{\theta+1}} \int u^\theta v^{2t-1} \gamma^2 \le  t^2 \int v^{2t-2} | \nabla v|^2 \gamma^2  + \int v^{2t} | \nabla \gamma|^2 + 2 t \int v^{2t-1} \gamma \nabla v \cdot \nabla \gamma.\]
  Multiply $ -\Delta v = u^\theta$ by $ v^{2t-1} \gamma^2$ and integrate by parts to obtain, after some rearrangement
  \[ t^2 \int v^{2t-2} | \nabla v|^2 \gamma^2 \le \frac{t^2}{2t-1}  \int u^\theta v^{2t-1} \gamma^2 - \frac{2 t^2}{2t-1} \int v^{2t-1} \gamma \nabla v \cdot \nabla \gamma.\]  We now use this to replace the first term on the right in the above inequality to obtain
  \[ \left( \sqrt{ \frac{p \theta (p+1)}{\theta+1}} - \frac{t^2}{2t-1} \right) \int u^\theta v^{2t-1} \gamma^2 \le \int v^{2t} | \nabla \gamma|^2 - \frac{t-1}{2(2t-1)} \int v^{2t} \Delta (\gamma^2),  \] and from this we easily get the desired result after considering the support of $ \gamma$ and how $ | \nabla \gamma|^2 $ and $ \Delta \gamma$ scale.

\end{proof}

In what follows we  shall need the following result,
 which is just an $L^1$ elliptic regularity result with the natural scaling.

\begin{lemma}  \label{reg}  For any integer $k \ge 0$ and $ 1 \le \alpha < \frac{N}{N-2}$ there is some $ C=C(k,\alpha)< \infty$ such that for any smooth $ w \ge 0$ we have
\[ \left( \int_{B_{R_k}} w^\alpha dx \right)^\frac{1}{\alpha} \le C
R^{2 + N ( \frac{1}{\alpha}-1)} \int_{B_{R_{k+1}}} | \Delta w| + C R^{ N( \frac{1}{\alpha}-1)} \int_{B_{R_{k+1}}} w.
 \]
\end{lemma}

We give a brief sketch of the proof even though the result is well known.

\begin{proof}  After a scaling argument it is sufficient to show there is some $C>0$ such that
\[ \left( \int_{B_1} w^\alpha dx \right)^\frac{1}{\alpha} \le C \int_{B_2} |\Delta w| + C \int_{B_2} w,\] for all smooth nonnegative $ w$.  Let $ 0 \le \phi \le 1$ denote a smooth cut off with $ \phi=1$ in $B_1$ and compactly supported in $ B_\frac{3}{2}$ and set $ v=w \phi$.  Then note that $L^1$ elliptic regularity theory gives  $ \| v\|_{L^\alpha(B_2)} \le C \| \Delta v\|_{L^1(B_2)}$  and writing this out gives
\[ \| w\|_{L^\alpha(B_1)} \le C \int_{B_2} |\Delta w| +  C \int_{B_2} w + C \int_{B_\frac{3}{2}} | \nabla w|,\] where $C$ is a changing constant independent of $w$.  To finish the proof we just need to control the first order term on the right.  We decompose $w$ as $ w=w_1 +w_2$ where $ \Delta w_1 = \Delta w$ in $B_2$ with $ w_1=0$ on $B_2$ and where $ w_2$ is harmonic in $ B_2$ with $ w_2=w$ on $ \partial B_2$.  Then by elliptic regularity theory $ \| \nabla w_1 \|_{L^1(B_2)} \le C \| \Delta w\|_{L^1(B_2)}$ and since $w_2$ is harmonic $ \| \nabla w_2 \|_{L^1(B_\frac{3}{2})} \le C \| w_2\|_{L^1(B_2)}$.   Combining these results gives
\[ \int_{B_\frac{3}{2}} | \nabla w| \le C \int_{B_2} |\Delta w| + C \int_{B_2} w + C \int_{B_2} |w_1|,\]  and the last term on the right can be controlled by $ \| \Delta w\|_{L^1(B_2)}$.  Recombining the results completes the proof.

\end{proof}

We will bootstrap the following result which is the key to obtain higher integral powers of $v$ controlled by lower powers and, as mentioned in the brief outline,  this is key in obtaining better nonexistence results.

\begin{prop} \label{initial}  Let $ (u,v)$ denote a positive stable solution of (\ref{eq}) with $ 1 \le p < \theta$.  Then for all $ 1 < \alpha < \frac{N}{N-2}$, $ t_0^- <t< t_0^+$ and nonnegative integers $k$ there is  some $  C< \infty$ such that for all $ R \ge 1$ we have
\begin{equation} \label{final}
\left( \int_{B_{R_k}} v^{ 2 t \alpha} \right)^\frac{1}{2 t \alpha} \le C R^{ \frac{N}{2t} ( \frac{1}{\alpha}-1)} \left(   \int_{B_{R_{k+3}}} v^{2t} \right)^\frac{1}{2 t}.
\end{equation}
\end{prop}

\noindent
\textbf{Proof of Proposition \ref{initial}.}   Let $t$ and $ \alpha$ be as in the hypothesis.   Set $ w=v^{2t}$ and note that
\[ |\Delta w| \le 2t (2t-1) v^{2t-2} | \nabla v|^2 + 2 t v^{2t-1} u^\theta,\] and also note that $ 2t-1 >0$ after considering the restrictions on $t$.  From Lemma \ref{reg} we have

\begin{eqnarray} \label{eee}
\left( \int_{B_{R_k}} v^{2t \alpha} \right)^\frac{1}{\alpha} & \le & C_t R^{2+N ( \frac{1}{\alpha}-1)} \int_{B_{R_{k+1}}} v^{2t-2} | \nabla v|^2  \nonumber \\
&&+ C_t R^{2+N( \frac{1}{\alpha}-1)} \int_{B_{R_{k+1}}} v^{2t-1} u^\theta \nonumber \\
&& + C_t R^{N(\frac{1}{\alpha}-1)} \int_{B_{R_{k+1}}} v^{2t}.
\end{eqnarray}  We begin by getting an upper bound on the gradient term.
 Let $ \phi$ denote a smooth cut off with $ \phi=1$ in $ B_{R_{k+1}}$ and compactly supported in $ B_{ R_{k+2}}$.  Multiply $ -\Delta v = u^\theta$ by $ v^{2t-1} \phi^2$ and integrate  by parts and apply Young's inequality to arrive at an inequality of the form
 \[ \int v^{2t-2} | \nabla v|^2 \phi^2 \le C \int u^\theta v^{2t-1} \phi^2 + C \int v^{2t} | \nabla \phi|^2,\] and after considering the support of $ \phi$ and how $ | \nabla \phi|^2$ scales with respect to $R$ we obtain
 \begin{equation} \label{part_1}
 \int_{B_{R_{k+1}}} v^{2t-2} | \nabla v|^2 \le C \int_{B_{R_{k+2}}} u^\theta v^{2t-1} + \frac{C}{R^2} \int_{B_{R_{k+2}}} v^{2t}.
 \end{equation}   Putting (\ref{part_1}) into (\ref{eee}) gives
  \begin{equation} \label{fff}
  \left( \int_{B_{R_k}} v^{2t \alpha} \right)^\frac{1}{\alpha} \le C R^{N( \frac{1}{\alpha}-1)+2} \int_{B_{R_{k+2}}} u^\theta v^{2t-1} + C R^{N( \frac{1}{\alpha}-1)} \int_{B_{R_{k+2}}} v^{2t}.
  \end{equation}  We now use Lemma \ref{one} to eliminate the first term on the right hand side of (\ref{fff}) to obtain
  \[  \left( \int_{B_{R_k}} v^{2t \alpha} \right)^\frac{1}{\alpha} \le C R^{N ( \frac{1}{\alpha}-1)} \int_{B_{R_{k+3}}} v^{2t},\]  and the proof is complete after raising both sides to the power $ \frac{1}{2t}$.

\hfill $ \Box$

If one performs an iteration argument of the above result and pays some attention to the allowable range of the various parameters they obtain the following.

\begin{cor}  \label{gone} \begin{enumerate} \item Suppose $ (u,v)$ is a smooth stable positive solution of (\ref{eq}) satisfying the hypothesis of Theorem \ref{MAIN}. Suppose  $ 1< p < \beta < \frac{2  N t_0^+}{N-2}$.  Then there is some integer $ n \ge 1$ and $ C<\infty$ such that
\begin{equation} \label{high}
 \left( \int_{B_{R}} v^{\beta} \right)^\frac{1}{\beta} \le C R^{N( \frac{1}{\beta}-\frac{1}{p})} \left( \int_{B_{R_{3n}}} v^p \right)^\frac{1}{p},
\end{equation}  for all $ 1 \le R$.

\item  Suppose that $ (u,v)$ is a stable smooth positive solution of (\ref{eq}) with $ p=1$.  For all $ 2 < \beta < \frac{2 N}{N-2} t_0^+$ there is some $ C<\infty$ and integer $n \ge 1$ such that
\begin{equation} \label{scalar_cor}
\left( \int_{B_R} v^\beta \right)^\frac{1}{\beta} \le C R^{ N( \frac{1}{\beta}-\frac{1}{2})} \left( \int_{B_{R_{3n}}} v^2 \right)^\frac{1}{2},
\end{equation} for all $ R \ge 1$.

\end{enumerate}

\end{cor}

\noindent
\textbf{Proof of Corollary \ref{gone}.}
  Let $ t_0^- <t_0 < t_0^+$ and let $ 1 \le \alpha_k < \frac{N}{N-2}$ and define $t_{k+1}=\alpha_k t_k$.   Iterating the result in Proposition \ref{initial} one obtains
 \begin{equation} \label{iter_gen}
 \left( \int_{B_{R_{3^n}}} v^{2 t_n \alpha_n} \right)^\frac{1}{2t_n \alpha_n} \le C R^{ \frac{N}{2} ( \frac{1}{t_n \alpha_n}- \frac{1}{t_0} )} \left( \int_{B_R} v^{2t_0} \right)^\frac{1}{2t_0},
  \end{equation} for all $ R \ge 1$ and all positive integers $n$ provided $ t_n < t_0^+$.  By suitably picking the $ \alpha_k$ for $ k \le n-1$ we see that $ 2 t_n \alpha_n$ can be made arbitrarily close to $ \frac{2N t_0^+}{N-2}$.     Note that when one performs the iterations that the powers of $R$ form a telescoping series and only the first and last terms don't cancel.

 We now separate the cases.  We first deal with case 1) and recall we are assuming the hypothesis from Theorem \ref{MAIN}.  So we either take $ 2 \le p < \theta$ and one can then show by a computation that $  t_0^- < \frac{p}{2}$  or we don't assume $ p \ge 2$ but we then, by hypothesis, assume $  t_0^- < \frac{p}{2}$.
 Also a computation shows that $ \frac{p}{2} <t_0^+$.    This allows one to pick $ t_0= \frac{p}{2}$.   With this choice of $ t_0$ and provided the above conditions hold on $ \alpha_k, t_k$ then (\ref{iter_gen}) gives
 \[    \left( \int_{B_{R_{3^n}}} v^{2 t_n \alpha_n} \right)^\frac{1}{2t_n \alpha_n} \le C R^{ \frac{N}{2} ( \frac{1}{t_n \alpha_n}- \frac{2}{p} )} \left( \int_{B_R} v^{p} \right)^\frac{1}{p}.\] This gives the desired result after considering the above comments on how big $2 t_n \alpha_n$ can be. One should note that we will only be interested in the case of $ \beta$ close to $\frac{2N}{N-2} t_0^+$.

 We now examine 2).  In this case we follow exactly the same argument as part 1)  but we now take $ t_0=1$,  which is allowed since $ t_0^- <1<t_0^+$.  Putting $t_0=1$ into (\ref{iter_gen}) gives the desired result.

\hfill $ \Box$

\noindent
\textbf{Completion of the proof of Theorem \ref{MAIN}.}   Suppose $(u,v)$ is a smooth positive stable solution of (\ref{eq}) and the hypothesis of Theorem \ref{MAIN} are satisfied.   Let $ p < \beta < \frac{2 Nt_0^+}{N-2}$.  Combining Corollary \ref{gone} 1) and Lemma \ref{Mid_lemma} there is some $ C< \infty $ such that
\[ \left( \int_{B_R} v^\beta \right)^\frac{1}{\beta} \le C R^{ N ( \frac{1}{\beta}- \frac{1}{p} ) + \frac{1}{p} ( N-2 - \frac{2 (p+1)}{p \theta-1})}, \]  for all $ R \ge 1$.  If this exponent is negative then after sending $ R \rightarrow \infty$ we obtain a contradiction.
Note the exponent is negative if and only if we have
\[ N < \frac{2 (\theta+1) \beta}{p \theta-1},\]  and after considering the allowable range of $ \beta$ we obtain the desired result.

\hfill $ \Box$

We now examine the case of the scalar equation; $p=1$.
 Critical to our approach in the following result:
\begin{lemma} \label{Wei_dong_init}  \cite{Wei_dong}  Suppose that $ (u,v)$ is a stable smooth positive solution of (\ref{eq}) with $ p=1$.  Then by Lemma \ref{equivalence}, $u$ is a positive stable solution of (\ref{fourth}) and then the results of \cite{Wei_dong} imply there is some $ C< \infty$ such that
\[ \int_{B_R} v^2 \le C R^{ N-4- \frac{8}{\theta-1}},\] for all $ R>0$.

\end{lemma}

To complete the proof of Theorem \ref{MAIN_four} we combine
Corollary \ref{gone} 2) and Lemma \ref{Wei_dong_init} and argue as in the proof of Theorem \ref{MAIN}. Our final result relates the usual notion of stability for the scalar equation to the systems notion of stability.

\begin{lemma}  \label{equivalence}   Suppose $(u,v)$ is a positive stable solution of (\ref{eq}) with $ 1=p<\theta$.  Then $u$ is a stable solution of (\ref{fourth}).

\end{lemma}

\begin{proof}     Let $ (u,v)$ be as in the hypothesis.  By definition there are smooth positive functions $ \zeta,\chi$ such that $ -\Delta \zeta = \chi,$ \; $  -\Delta \chi = \theta u^{\theta-1} \zeta$ in $ \IR^N$.    Let $ \gamma $ be smooth and compactly supported.    First note that $ -\Delta \zeta > 0$ and $ \Delta^2 \zeta = \theta u^{\theta-1} \zeta$ and so we
\begin{eqnarray*}
\int \theta u^{\theta-1} \gamma^2 &=& \int \Delta^2 \zeta ( \gamma^2 \zeta^{-1}) \\
&=& \int \Delta \zeta \Delta ( \gamma^2 \zeta^{-1}) \\
&=& 2 \int \frac{\Delta \zeta}{\zeta} | \nabla \gamma|^2 + 2 \int \frac{\Delta \zeta}{\zeta} \gamma  \Delta \gamma \\
&&+ 2 \int (\Delta \zeta) \frac{\gamma^2 | \nabla \zeta|^2}{\zeta^3} - \int \frac{ (\Delta \zeta)^2}{\zeta^2} \gamma^2 +I
\end{eqnarray*}  where
\[ I= -4 \int (\Delta \zeta) \gamma \frac{ \nabla \gamma \cdot \nabla \zeta}{\zeta^2}.\]   Using Young's inequality and the fact that $ -\Delta \zeta \ge 0$ we see that
\[ | I | \le -2 \int \frac{\Delta \zeta  | \nabla \gamma|^2}{\zeta} - 2 \int \frac{\gamma^2 | \nabla \zeta|^2 \Delta \zeta}{\zeta^3}.\]  Using this upper bound we see that
\[ \int \theta u^{\theta-1} \gamma^2 \le  2 \int \frac{ \gamma \Delta \zeta}{\zeta} \Delta \gamma - \int \frac{(\Delta \zeta)^2}{\zeta^2} \gamma^2,\]  and this is bounded above, after using Young's inequality again,  by  $ \int (\Delta \gamma)^2$,  which is the desired result.

\end{proof}

\section{Results on the half space}

In this section we are interested in Liouville theorems on the half space.   For notational convenience all integrals will be over $ \IR^N_+$ unless otherwise indicated.

Our first result shows that monotonic solutions on the half space satisfy the stability like inequality given by  (\ref{second}) but in fact note we prove slightly more.  The test functions need not be zero on the boundary of $ \IR^N_+$.

\begin{lemma} \label{toboun} Suppose $f,g$ are sufficiently smooth, positive increasing nonlinearities on $ (0,\infty)$ with $ f(0)=g(0)=0$.  Suppose $(u,v)$ is a positive solution of   $ -\Delta u = f(v)$, $ -\Delta v=g(u)$ in $ \IR^N_+$ which satisfies $ u_{x_N},v_{x_N} >0$ in $ \IR_N^+$.  Then
\[ \int_{\IR^N_+} \sqrt{f'(v) g'(u)} \phi^2 \le \int_{\IR^N_+} | \nabla \phi|^2 \qquad \forall \phi \in C^2_c(\IR^N).\]   Note the test functions need not be zero on $ \partial \IR_+^N$.
\end{lemma}

\begin{proof} By taking a derivative in $ x_N$ of (\ref{system_thm_half})  we see that $ -\Delta u_{x_N} = f'(v) v_{x_N}$ and $ -\Delta v_{x_N}=g'(u) u_{x_N}$ in $ \IR^N_+$.
 Let $ \phi \in C_c^\infty(\IR^N)$ and multiply the first equation by
$ \frac{\phi^2}{ u_{x_N}}$ and the second equation by $ \frac{\phi^2}{v_{x_N}}$ (and note by Hopf's Lemma that $u_{x_N},v_{x_N}>0$ on $ \partial \IR^N_+$) and integrate over the half space to obtain
\[ \int \frac{f'(v) v_{x_N} \phi^2}{u_{x_N}} \le  \int \nabla (u_{x_N}) \cdot \nabla ( \phi^2 u_{x_N}^{-1}) - \int_{\partial \IR^N_+} \partial_\nu u_{x_N} ( \phi^2 u_{x_N}^{-1}),\]
\[ \int \frac{g'(u) u_{x_N} \phi^2}{v_{x_N}} = \int \nabla (v_{x_N}) \cdot \nabla ( \phi^2 v_{x_N}^{-1}) - \int_{\partial \IR^N_+} \partial_\nu v_{x_N} ( \phi^2 v_{x_N}^{-1}),\] where $ \partial \nu$ is the outward pointing normal.
   We first examine the boundary integrals.  Note that $ \partial_\nu u_{x_N}= - u_{x_N x_N}$ and from the fact that $u,v$ are sufficiently regular to the boundary we see that
$ -u_{x_N x_N}= f(v) + \sum_{k=1}^{N-1} u_{x_k x_k} $ on $ \IR_+^N$.  Now note that by the assumption of $f$ and the boundary condition on $u$ we see that we must have $ u_{x_N x_N}=0$ on $ \partial \IR^N_+$ and hence the boundary integral is zero.  Similarly one shows the other boundary integral is also zero.      We then use Young's inequality on the integrals involving the gradients and add the results to see that
\[ \int \left(  \frac{f'(v) v_{x_N} }{u_{x_N}} + \frac{g'(u) u_{x_N} }{v_{x_N}}  \right)  \phi^2  \le 2 \int | \nabla \phi|^2.\]  We then use the argument from Lemma \ref{stabb} to obtain the desired result.   \\
An alternate proof can be given by extending the solutions and the nonlinearities to $ \IR^N$, using odd extensions.  One then has that $(u,v)$ is a monotonic solution of the extended problem on $\IR^N$ and hence is stable.   Fix $ \phi$ to be a smooth and compactly supported function in $ \IR^N_+$ but we allow $ \phi$ to be non zero on the boundary and then extend $ \phi$ to all $ \IR^N$ using an even extension.  The extension is a sufficiently regular test function which can be inserted into  (\ref{second}) and this gives the desired result after writing all the integrals over the half space.
\end{proof}

  \textbf{Proof of Theorem \ref{system_thm_half}.}
  Suppose $(u,v)$ is a bounded positive classical  solution of (\ref{eq_half}).   By a moving plane argument, see \cite{dancer_2, sirak}  one has $ u_{x_N}, v_{x_N}>0$ in $ \IR^N_+$.    For $x \in \IR^N$ we write $ x=(x',x_N)$ and we now define,  for each $ t>0$,  $ u_t(x)=u(x', x_N+t)$ and $ v_t(x)=v(x', x_N+t)$ for $x \in \IR_N^+$.    Note that $u_t,v_t$ are monotonic solutions of (\ref{eq_half}) in $ \IR^N_+$ but without any assumptions on the boundary values.    Using the same argument as in the proof of Lemma \ref{toboun} one can easily show that
   \begin{equation} \label{shoo}
   \sqrt{p \theta} \int v_t^\frac{p-1}{2} u_t^\frac{\theta-1}{2}  \le \int | \nabla \phi|^2, \qquad \forall \phi \in C_c^\infty( \IR^N_+).
      \end{equation}
      Now note that since $u$ and $v$ are bounded, monotonic and positive we see that
  \[ w_1(x'):=\lim_{t \nearrow \infty} u_t(x), \qquad w_2(x'):= \lim_{t \nearrow \infty} v_t(x),\] defined in $ \IR^{N-1}$ are positive bounded solutions of
\begin{equation}
-\Delta w_1 = w_2^p, \qquad -\Delta w_2 = w_1^\theta, \qquad \mbox{ in } \IR^{N-1}.
\end{equation}  To complete the proof we will show that $ w_1,w_2$ preserve some of the stability properties of the solutions on the half space.  We won't show $w_1,w_2$ are stable solutions but we will instead show that $w_1,w_2$ satisfy the stability like inequality given by  (\ref{second}) on $ \IR^{N-1}$.   One can then apply Theorem  \ref{MAIN} (note the only place stability is used in the proof of Theorem \ref{MAIN} is to obtain (\ref{second})).

   Let $ \phi_1 \in C_c^\infty(\IR^{N-1})$ and let $0 \le \phi_R \le 1$ be smooth and compactly supported  in $ (R,4R) \subset \IR$ where $ R \ge 1$ with $ \phi_R =1$ in $ (2R,3R)$. Note there is some positive finite $C$ such that $ | \phi'(x_{n}) | \le \frac{C}{R}$ for all $  R \ge 1$ and $ x_N \in (R,4R)$.   Define $ \phi(x)= \phi_1(x') \phi_R(x_N)$ and putting $ \phi $ into (\ref{shoo}) and using the fact that $ u_t(x) \ge u(x',t)$ and similarly for $v$ shows (after writing the integrals as iterated integrals then using some algebra) that
\begin{eqnarray*}
\sqrt{ p \theta} \int_{\IR^{N-1}} v(x',t)^\frac{p-1}{2} u(x',t)^\frac{\theta-1}{2} \phi_1(x')^2 d x' & \le & \int_{\IR^{N-1}} | \nabla \phi_1(x')|^2 d x' \\
&& + T_R \int_{\IR^{N-1}} |  \phi_1(x') |^2 d x'
\end{eqnarray*}

where
\[T_R:= \frac{ \int_{\IR} | \nabla \phi_R(x_N)|^2 d x_N   }{   \int_{\IR} \phi_R(x_N)^2 d x_N}.\]   One easily sees that $ T_R \rightarrow 0$ as $ R \rightarrow \infty$ and so sending $ R \rightarrow \infty$ and then sending $ t \rightarrow \infty$ gives the desired result.
  \hfill $\Box$

We now give some comments on the above proof.  Firstly,  this idea of relating a monotonic problem on the half space to a problem in one dimension lower on the full space has been used by many authors.     A key step in the above argument is to show the problem in $ \IR^{N-1}$ is stable.
 We first learned of this idea of showing the limiting solution is stable in the recent work \cite{new_far} where the context was a quasilinear scalar problem.  After examining the literature we realized this result is contained in \cite{Wei_dong} where they are examining a biharmonic problem.

\end{document}